\documentclass[a4paper,12pt]{article}

\usepackage[english]{babel}
\usepackage{amssymb,amsfonts,amsmath,amsthm,mathtext}
\usepackage{cite,enumerate,float,indentfirst}
\usepackage[pdftex]{hyperref}
\usepackage{graphics}

\usepackage[centerlast,small]{caption2}

\title{Compactification of the space of branched coverings of the two-dimensional sphere}
\author{Zvonilov V.I, Orevkov S.Yu.\thanks{The second author was supported by the RSF grant, 
	project 14-21-00053 dated 11.08.14.}}

\begin{document}

\newcommand{\Cb}{\mathbb{C}}

\newtheorem{te}{Theorem}
\newtheorem{pr}{Proposition}
\newtheorem{lem}{Lemma}
\newtheorem{sle}{Corollary}
\newtheorem{za}{Remark}
\newtheorem{opr}{Definition}
\newtheorem{ex}{Example}

\renewcommand{\captionlabeldelim}{.}

\maketitle

\begin{abstract}
For a closed oriented surface  $ \Sigma $ we define its degenerations into singular surfaces that are locally homeomorphic to wedges of disks. Let  $X_{\Sigma,n}$ be the set of isomorphism classes of  orientation preserving $n$-fold branched coverings  $ \Sigma\rightarrow S^2 $ of the two-dimensional sphere. We complete $X_{\Sigma,n}$ with the  isomorphism classes of mappings that  cover the sphere by the degenerations of $ \Sigma $. In case $ \Sigma=S^2$,  the topology that we  define on the obtained completion  $\bar{X}_{\Sigma,n}$ coincides  on $X_{S^2,n}$ with the topology induced  by the space of coefficients of rational functions $ P/Q $, where $ P,Q $ are homogeneous polynomials of degree $ n $ on $ \Cb\mathrm{P}^1\cong S^2$.

 We prove that $\bar{X}_{\Sigma,n}$ coincides with the Diaz-Edidin-Natanzon-Turaev compactification of the Hurwitz space $H(\Sigma,n)\subset X_{\Sigma,n}$ consisting of isomorphism classes of  branched coverings with all critical values being simple. 
\end{abstract}

 \section{Introduction}
Suppose  $ \Sigma $ is a closed oriented surface being fixed throughout the paper. Orientation preserving  $ n $-sheeted branched coverings $f_1, f_2:\Sigma\rightarrow S^2$ of $ 2 $-sphere are \emph{isomorphic} if there exist a homeomorphism $ \alpha :\Sigma\rightarrow \Sigma $ with $f_2\circ\alpha=f_1$.
Let $X_{\Sigma,n}$ be the set of isomorphism classes of such coverings, and  $ H(\Sigma,n)\subset X_{\Sigma,n}$ be the subset of isomorphism classes of coverings  with simple critical values. Due to Riemann-Hurwitz formula for a covering $ f\in X_{\Sigma,n} $, the sum of multiplicities of critical values is $2n-\chi(\Sigma)$, where $ \chi $ is the Euler characteristic.

Denote by $ P^k $ the symmetric power of $ S^2 $, i.e. the quotient space of  $ (S^2)^k $ by the symmetric group. For $ k=2n-\chi(\Sigma) $, let  $ LL:X_{\Sigma,n} \rightarrow P^k$ be the map that takes  the isomorphism class  $ [f] $ of a covering $ f $ to the set of critical values of the covering with their  multiplicities. 

As it is well known (see, e.g, \cite[5.3.5]{LZ}),  $ H(\Sigma,n)$ admits a natural topology such that the restriction  $ LL|_{H(\Sigma,n)}$ is an unbranched covering of its image.  $ H(\Sigma,n) $ endowed with this topology is called \emph{Hurwitz space}.

Compactifications of $H(\Sigma,n)$ were discussed
by different authors starting with
Harris and Mumford's seminal paper \cite{HM}. Natanzon and
Turaev \cite{NT} constructed a compactification $ N(\Sigma,n) $ of $ H(\Sigma,n)$ by topological methods.  Diaz \cite{D} proved that  $ N(\Sigma,n) $ coincides with the compactification of  $ H(\Sigma,n) $ built in \cite{DE} by means of algebraic geometry.

We think that our description of the compactification $ N(\Sigma,n) $ of the Hurwitz space is geometrically more illustrative. Moreover, the description makes it clear that $X_{\Sigma,n}$ is naturally embedded in  $ N(\Sigma,n) $. In particular, we compactify the space $\mathcal{X}_n$ of isomorphism classes of rational functions on $ \Cb\mathrm{P}^1\cong S^2$. 
 Our main interest is the subspace of  $\mathcal{X}_n$ composed of  $j$-invariants of trigonal curves on a ruled surface, i.e. of the maps from the base of the  ruled surface to the modular curve, see  \cite[n. 4]{Or}, and also \cite[2.1.2, 3.1.1]{Degt}. In further papers, we intend to apply the obtained results to compute the fundamental group of the space of nonsingular trigonal curves.

Hurwitz spaces are sometimes understood also as more general objects in which the base $ B $ of $ n $-sheeted branched coverings is not $ S^2 $ but a closed oriented surface of any genus. One may also require that
the local monodromy at all critical values should belong to fixed conjugacy classes in some permutation group on $n$ points. For $B=S^2$, these spaces can be identified with the strata of some natural stratification of the pair  $(N(\Sigma,n),X_{\Sigma,n})$. See \cite{BE}~ -- \cite{K}.

The paper is organized as follows. In Sections \ref{singsurf}, \ref{Deform} we introduce the notions of degenerate (singular) surface and  its covering map onto the $ 2 $-sphere. In Section \ref{Xn} we complete  $X_{\Sigma,n}$  with  isomorphism classes of mappings that  cover the sphere by  degenerations of $ \Sigma $. We  define a topology on the obtained completion $\bar{X}_{\Sigma,n}$  and we use an extension of  $ LL $ to  $\bar{X}_{\Sigma,n}$ to prove that the topology is Hausdorff. In Section \ref{rationf} we prove that  $X_{S^2,n}$ is homeomorphic to the space of isomorphism classes of rational functions on $\Cb\mathrm{P}^1$. In Section \ref{NTc} we prove that  $\bar{X}_{\Sigma,n}$ and $ N(\Sigma,n) $ are homeomorphic, the key reference being Proposition \ref{pert} asserting that if the boundary of a connected oriented surface covers a circle  then there exists a unique, up to isotopy,   extension of the covering to a branched covering of the disk (see \cite[Proof of Corollary 2.2]{Nat}). 

\section{The space of isomorphism classes of branched coverings}

\subsection{Singular surfaces}
	\label{singsurf}
We shall use a topological counterpart to the notion of singular algebraic curve. 

 In this paper,  \emph{a surface}  is a Hausdorff topological space with a countable base such that any its point has a neighbourhood homeomorphic to an open disk  or to a wedge of open disks (or to the union of an open half-disk  with its diameter in case of a surface with boundary).  We call such a neighbourhood \emph{admissible} and say that the disks of the wedge are \emph{the branches of the surface} at the center of the wedge.
We assume that a non-negative integer 
 $ g_v$ (called  \emph{the local genus of the surface at the point} $v$) is assigned to any point $v$ of such a surface; the local
 genus is non-zero only at a finite number of points and it vanishes on the boundary.  That is,  speaking more formally, a surface is a pair of a topological space  and an integral-valued function  $v\mapsto g_v$ on it such that the space and the function have the above properties.  
Let $m_v$ be the number of branches of a surface at the point  $ v $ and  $ \mu_v=2g_v+m_v-1$ be  \emph{the Milnor number} of $ v $.
  Points with  
  $ \mu_v>0$ are \emph{the singular points} of the surface. A surface  is 
 \emph{degenerate} (or \emph{singular}) if is has
 singular points. Otherwise it is  \emph{nonsingular}. 
 
 The \emph{the normalization} of a surface is defined as a nonsingular
 surface obtained by replacing an admissible neighbourhood of each
 singular point with a disjoint union of smooth disks. There is a natural projection of the normalization to the surface that takes each glued disk  to the corresponding disk  of the wedge. 
  \emph{A component} of a surface is a connected component of its normalization. So    \emph{a component of a surface is a two-dimensional manifold.} A surface is \emph{oriented/closed} if its normalization is oriented/closed (compact and without  boundary).
  
Below all the surfaces are oriented and all the maps are
orientation preserving. 

Let $ \Sigma_0 $ be a surface (in the above sense, maybe singular), $ v $ be its point, and
$U_v$ be the closure of an admissible neighbourhood of $v $.
Take a compact connected orientable  (maybe singular) surface 
  $\tilde{U}_v$ with $m_v$ boundary components and such that
$b_1(\tilde{U}_v)+\sum_{x\in\tilde{U}_v}\mu_x=\mu_v$ and $b_2(\tilde{U}_v)=0$ (i.e. $\tilde{U}_v$ has no components without boundary);
here $b_i$ is $i$-th Betti number. Glue  the surfaces $ \Sigma_0\setminus\mathrm{Int}U_v $ and $\tilde{U}_v$ along their common boundary. We say that the resulting surface 
 $ \Sigma_1$ is \emph{a perturbation} of $ \Sigma_0$ \emph{at the point} $ v $, or \emph{a local perturbation}, and $ \Sigma_0$ is \emph{a local degeneration} of $ \Sigma_1$, \emph{by means of
} $U_v$, $\tilde{U}_v$.
A composition of a finite number of local perturbations/degenerations of a surface is called a \emph{perturbation/degeneration} of the surface.

\begin{za}
	\label{mu}
For the surface $\tilde{U}_v$ from the above definition,  $b_1(\tilde{U}_v)=1-\chi(\tilde{U}_v)$ and if $ \tilde{U}_v $ is nonsingular then $\mu_v=1-\chi(\tilde{U}_v)$.     
\end{za}
\begin{proof}
The first equality holds for any surface 
$\tilde{U}_v$ with the listed  properties. The second one is evident from the definition of perturbation.	
\end{proof}

 \subsection{Perturbations/degenerations of coverings}
\label{Deform}
A map $ f $  of a surface $ \Sigma' $ to the $ 2 $-sphere is called  \emph{a branched covering} if, for $ \pi $ being the projection of normalization, the restriction of the composition $f\circ\pi$  to any component of the surface is an orientation preserving branched covering and each singular point   $ v \in\Sigma' $ with $ m_v=1 $  
 is a ramification point of 
  $f$. Ramification points of $f$ and  singular points of
  $\Sigma'$ are called  \emph{critical points} of $f$ and their images are called
 \emph{critical values}.  By $ \mathrm{cr} f $ we denote the set of all critical values of $ f $.

Let $v$ be a critical point of a branched covering $f_0:\Sigma_0\rightarrow S^2$
and let $\bar{U}_w\subset S^2$ be a closed disk which contains
$w=f_0(v)$ as an inner point and which does not contain other critical
values of $f_0$. 
Pick a connected component  $U_v$  of $f^{-1}(\bar{U}_w)$ with $v\in U_v$. Clearly it is a closed admissible neighbourhood of  $v$.
Let $ \Sigma_1 $ be a perturbation of $ \Sigma_0 $  at $ v $ by means of
 $U_v$, $\tilde{U}_v$. We say that a branched covering  $f_1:\Sigma_1\rightarrow S^2$ is  \emph{a perturbation} of $f_0$ \emph{at} $ v $, or \emph{local perturbation with  perturbation domain} $ \bar{U}_w $ and that  $f_0$ is  \emph{a local degeneration} of $f_1$ if  $f_0=f_1$ on $\Sigma_0\setminus  \mathrm{Int}U_v=\Sigma_1\setminus\mathrm{Int}\tilde{U}_v$. It is obvious that $ f_1(\tilde{U}_v)=\bar{U}_w $.

 Now suppose that the disks $\bar{U}_w$ picked for all critical values $ w $ of  $ f_0$ are disjoint.  
A composition $ f$ of a finite number of local perturbations of $ f_0$ that have the perturbation domains $ \bar{U}_w $ is called   \emph{a perturbation} of the covering. The union $V=\bigcup_w\bar{U}_w$ is \emph{a perturbation domain} of $f_0$. In this case  $ f_0 $ is called \emph{a degeneration} of $ f $.

The following Remark is immediate from the definition of perturbation.

\begin{za}
	\label{trans} \emph{Transitivity of perturbation:}
	Let  $ f_1:\Sigma_1\rightarrow S^2$ be a perturbation of a branched covering  $ f_0:\Sigma_0\rightarrow S^2$ with the perturbation domain $ V_0 $ and   $ f_2:\Sigma_2\rightarrow S^2$ be a perturbation of a branched covering $ f_1$ with the perturbation domain $ V_1\subset V_0$. Then  $ f_2$ is a perturbation of $ f_0$ with the perturbation domain $V_0$.     
\end{za}
Below, we need degenerations only of the fixed   \emph{nonsingular} surface $ \Sigma$ 
 (see the introduction). 
 
Let  $ f':\Sigma'\rightarrow S^2$ be a degeneration of a branched covering $ f:\Sigma\rightarrow S^2$ corresponding to a perturbation domain $ V $. For any critical point  
$ v$ of $ f'$, let $ U_v\subset\Sigma',\tilde{U_v}\subset\Sigma$ be the corresponding surfaces involved in the definition
of perturbation.  
Denote by $k_1,k_2,\ldots,k_{m_v}$ the degrees of the restrictions of $ f' $ to the boundary circles of   $U_v$. It is clear that  $\sum_{i=1}^{m_v}k_i$ is equal to the local degree $ \mathrm{deg}_vf' $ of $ f' $ at the point $ v $ that is the multiplicity of the point as a root of the equation $w=f'(v)$. 
\emph{The index of a point} $ v\in \Sigma' $ is defined as $\mathrm{ind}_v=\mu_v + \mathrm{deg}_vf'$. 

\emph{The multiplicity of a critical value} $ w $ of  $ f' $ is $ \sum_{v\in f'^{-1}(w)}(\mathrm{ind}_v-1) $. A critical value is called \emph{simple} if it is of multiplicity $1$. Otherwise it is called \emph{multiple}. 
 
\begin{pr}
	\label{R-H}
	Let  $ f':\Sigma'\rightarrow S^2$ be a degeneration of a branched covering $ f:\Sigma\rightarrow S^2$. Then the sum of the multiplicities of critical values of $f'$ does not depend on $f'$
	and it is equal to $2n-\chi(\Sigma)$, where $n=\mathrm{deg}f=\mathrm{deg}f'$ is the covering degree.
\end{pr}
\begin{proof}
	Let $ f $ be obtained from $ f' $ with a perturbation domain $ V $. For any critical point $ v $ of $ f' $, let $\tilde{U}_v\subset\Sigma$ be the corresponding connected component of $ f^{-1}(V) $. The index $ \mathrm{ind}_v $ is equal to $1-\chi(\tilde{U}_v)+ \mathrm{deg}_vf'$ due to Remark \ref{mu}. Let $ cr=\mathrm{cr} f'$ and let $|cr|$ be the number of the critical values. The sum of their multiplicities equals $$ \sum_{w\in cr  }\sum_{v\in f'^{-1}(w)}(\mathrm{ind}_v-1)=\sum_{w\in cr  }\sum_{v\in f'^{-1}(w)}(\mathrm{deg}_vf'-\chi(\tilde{U}_v))=n|cr  |-\sum_v\chi(\tilde{U}_v).$$ Since  $ \chi(S^2\setminus V)=2-|cr  |$ and $ f $ is an unbranched covering over $ S^2\setminus V$, we have $$n|cr  |-\sum_v\chi(\tilde{U}_v)=2n-\chi(f^{-1}(S^2\setminus V))-\sum_v\chi(\tilde{U}_v)=2n-\chi(\Sigma).$$	
\end{proof} 

\begin{pr}
	\label{degen}
For any branched covering $ f:\Sigma\rightarrow S^2$ and for any partition of $\mathrm{cr}f $, there exists a degeneration $ f' $ of $ f $ such that any critical value of $ f' $ corresponds to a partition class and has a multiplicity which is equal to the sum of multiplicities of all critical values in this class. 
\end{pr}
\begin{proof}
	
	 We chose pairwise disjoint topological closed disks is $S^2$ so that the interior of each disk contains
	 exactly one class of the partition. In each disk we pick an interior point.
	 Let $\Sigma'$ be a surface obtained from $\Sigma$ by replacing each connected component $\tilde U_v$
	 of the preimage of each chosen disk $\bar U_w$ ($w$ being the picked point in it)
	 with a wedge of disks $U_v=\bigvee_{i=1}^{m_v}D_i$ where $m_v$ is the number of boundary components of
	 $\tilde U_v$. 
	 It is clear that $ \Sigma'$ is a degeneration  of  $ \Sigma $. For the restriction of $ f $ to  $ \Sigma\setminus\bigcup_w\mathrm{Int}\bar{U}_w $, let $ f':\Sigma'\rightarrow S^2$ be its extension such that the restriction $f'$ to $D_i$
	 is a branched covering over $\bar{U}_w$ of the corresponding multiplicity with a unique critical  point $ v \in f'^{-1} (w) $. Since $ \Sigma $ is nonsingular, Remark \ref{mu} combined with Riemann-Hurwitz formula for  $f:f^{-1}(\bar{U}_w)\rightarrow\bar{U}_w$
	 (and with the definition of multiplicity applied to the critical value $w$) implies the required
	 statement about the multiplicities. 
\end{proof}

Suppose that the critical values of  branched coverings  $f_1:\Sigma_1\rightarrow S^2$ and $f_2:\Sigma_2\rightarrow S^2$ coincide together with their multiplicities. Suppose also that there exists an isotopy  $ \{\varphi_t:S^2\rightarrow S^2\}_{t\in[0,1]} $ of the identity mapping  such that for all $ t\in[0, 1] $ the homeomorphism $ \varphi_t $ fixes every critical value and there exists  an orientation preserving  homeomorphism $\beta:\Sigma_1\rightarrow \Sigma_2$   with $f_2\circ\beta=\varphi_1\circ f_1$. Then these coverings are called \emph{isomorphic}.
\begin{pr}
	\label{equiv}
		\emph{(Cf. \cite[Proof of Lemma 1.3.1]{NT})}.  Branched coverings  $f_1$ and $f_2$ are isomorphic if and only if there exists a homeomorphism
	$\alpha:\Sigma_1\to\Sigma_2$ such that $f_2\circ\alpha = f_1$. \qed
\end{pr}
 Thus the new definition of isomorphic coverings is equivalent to the one given in the introduction 
(notice that the definition from the introduction  can be extended word by word to singular covering surfaces).

It is clear that if two coverings are isomorphic then for any perturbation/degeneration $ f $ of  one of them there is a perturbation/degeneration of another one isomorphic to $ f $. So we can speak about a perturbation/degeneration of an  isomorphism class  of covering.

Let $V=\bigcup_{e\in E} D_e\subset S^2$ be a union of disjoint closed topological disks (here $e$ is a fixed interior point of the disk $D_e$).  Suppose that all the critical values of branched coverings  
 $ f_0:\Sigma_0\rightarrow S^2$ and $ f_1:\Sigma_1\rightarrow S^2$ are simple and lie in $\mathrm{Int}\, V $.
We say that these coverings are    $(V,E)$-\emph{equivalent} if there exist homeomorphisms    
 $ \alpha:\Sigma_0\rightarrow \Sigma_1$ and $\varphi:S^2\rightarrow S^2$ such that:   
\renewcommand{\labelenumi}{(\theenumi)}
\renewcommand{\theenumi}{\roman{enumi}}
\begin{enumerate}
	\item $ \varphi\circ f_0=f_1\circ\alpha $;
	\item  $ \varphi $ leaves the points of  $E\cup~(S^2~\setminus~\mathrm{Int}V)$ fixed and it is isotopic to the identity mapping in the class of such homeomorphisms leaving the points fixed.
	 
\end{enumerate}  
\begin{pr}
	\label{pert}
		For any branched covering $ f_0:\Sigma_0\rightarrow S^2$ there exists its perturbation $f_1:\Sigma_1\rightarrow S^2$  at a critical point $ v $ such that the surface  $\tilde{U}_v\subset \Sigma_1$ (involved in the definition of perturbation)  is nonsingular and  $\chi(\tilde{U}_v)=2-2g_v-m_v$. Moreover, $ f_1$ can be chosen so that all of its critical values  belonging to the perturbation domain $ \bar{U}_w $, $ w=f_0(v) $, are simple and  can be placed at any prescribed positions;
		in this case the covering $ f_1 $ is unique up to $ (\bar{U}_w,w) $-equivalence.
\end{pr}
\begin{proof}
 Consider a connected nonsingular surface $\tilde{U}_v$ of genus
$g_v$ with $m_v$ boundary components. We have   $\chi(\tilde{U}_v)=2-2g_v-m_v$. 
	Let $f:\tilde U_v\to\bar U_w$ be a branched covering with critical values at the prescribed positions such that  $f=f_0$ on $\partial U_v = \partial\tilde U_v$. To construct the covering it is sufficient to decompose 
	 $m_v$ disjoint cycles (the monodromy of the covering along  $\partial\bar U_w$) into a product of transpositions that generate a transitive group of permutations (see, e.g. 	\cite[Corollary 10.12]{Degt}). 
Let  $\Sigma_1=(\Sigma_0\setminus U_v)\cup\tilde{U}_v$ and put  $ f_1=f_0 $ on $\Sigma_0\setminus U_v$ and $ f_1=f $ on $\tilde{U}_v$. 

The uniqueness follows from \cite[Proof of Corollary 2.2]{Nat} (see also 
	\cite[Corollary 10.12]{Degt}). The proof of the fact that the corresponding isotopy fixes $ w $ uses the lifting the isotopy to the covering space (cf. \cite[Proof of Lemma 1.3.1]{NT}).
\end{proof}

\subsection{Topology on the set of isomorphism classes of branched coverings}
\label{Xn}
Remind that $ \Sigma $ is a fixed nonsingular oriented closed surface. As in  the introduction, let  $X_{\Sigma,n}$ denote the set of isomorphism classes of $ n $-folded branched coverings $ \Sigma\rightarrow S^2 $ and $\bar{X}_{\Sigma,n} $ denote the set of all degenerations of isomorphism classes of coverings from $X_{\Sigma,n}$. 

Given  $f$ with $ [f]\in\bar{X}_{\Sigma,n} $ and a perturbation domain $ V $ of $ f $,  
let $ U_{[f],V}\subset\bar{X}_{\Sigma,n} $ be the set  of isomorphism classes of all perturbations of $ f $ whose perturbation domain is $ V $. In other words, these are the classes of perturbations of coverings in $ [f] $ such that all their critical values belong to $ V $.
\begin{pr}
	\label{base} 
	The family  $$
	\{U_{[f],V}\mid \text{$f\in\bar{X}_{\Sigma,n}$ \emph{and $V$ is a perturbation domain of} $f$}  \}
	$$
	is a base of topology on  $\bar{X}_{\Sigma,n} $. The subset $X_{\Sigma,n}$ is dense in the space  $\bar{X}_{\Sigma,n} $ endowed with this topology.	    
\end{pr}
\begin{proof}
	It is clear that the union of all members of the family coincides with  $\bar{X}_{\Sigma,n} $. Further, if $[f]\in U_{[f_1],V_1}$ then for a perturbation domain  $ V\subset V_1$ of $ f $ we have  $U_{[f],V}\subset U_{[f_1],V_1}$ by the transitivity of perturbation (see Remark \ref{trans}). Thus, if $[f]\in U_{[f_1],V_1}\cap  U_{[f_2],V_2} $ then for a perturbation domain $ V\subset V_1\cap V_2 $  of $ f $ we have $U_{[f],V}\subset U_{[f_1],V_1}\cap  U_{[f_2],V_2} $. Hence $ \{U_{[f],V}\}$ is  a base of topology on  $\bar{X}_{\Sigma,n} $.
	
	The fact that $ X_{\Sigma,n} $ is dense immediately follows from the definition of the topology on $\bar{X}_{\Sigma,n} $.
\end{proof}
Recall that $ H(\Sigma,n)\subset X_{\Sigma,n}$ is the subset consisting of  isomorphism classes of  coverings with  simple critical values (see the introduction).  
\begin{pr}
	\label{nspert}
	For any branched covering $ f_0 $ with $ [f_0]\in\bar{X}_{\Sigma,n} $ and for its perturbation domain $ V $ there is a class  $[f]\in U_{[f_0],V}\cap H(\Sigma,n)$, where 
	$ f $ is unique up to  $(V,\mathrm{cr}  f_0)$-equivalence.
\end{pr}
\begin{proof}
	Let $ v $ be a critical point of $ f_0 $ and $ \bar{U}_w$ be a connected component of $ V $, which is a closed neighbourhood of the point $w=f_0(v)$. Due to Proposition \ref{pert}, there is a perturbation $[f_v]\in U_{[f_0],V}$ such that its critical points lying in $ \bar{U}_w$ are simple, the surface $\tilde{U}_v$ (the connected component of $f^{-1}(V) $ corresponding to $ v $) is nonsingular, and $\chi(\tilde{U}_v)=2-2g_v-m_v=1-\mu_v$. Let $f:\Sigma_1\rightarrow S^2$ be the composition of all the local perturbations $f_v$.  	
	
	Since $ [f_0]\in\bar{X}_{\Sigma,n} $, there is a perturbation $ f' $ of $ f_0 $ with $[f']\in X_{\Sigma,n}$ and with a perturbation domain $ V'$. Let us prove that the surfaces $ \Sigma $ and $ \Sigma_1 $ are 	homeomorphic. 
	Obviously, $ S^2\setminus \mathrm{Int}V $ is homeomorphic to $ S^2\setminus \mathrm{Int}V'$ and $ f_0 $ is an unbranched covering. Hence, by the definition of perturbation of branched covering, we have $f^{-1}(S^2\setminus \mathrm{Int}V)=f_0^{-1}(S^2\setminus \mathrm{Int}V)\cong f_0^{-1}(S^2\setminus \mathrm{Int}V')=f'^{-1}(S^2\setminus \mathrm{Int}V') $.
	Let $\tilde{U}'_v$ be a connected component of $f'^{-1}(V') $ corresponding to a critical point $ v $ of $ f_0 $. The above arguments combined with 
	 Remark \ref{mu}  yield  $\chi(\tilde{U'}_v)=1-\mu_v=\chi(\tilde{U}_v)$. The boundaries of $\tilde{U}_v$, $\tilde{U}'_v$ consist of $ m_v $ circles. Thus $\tilde{U}_v\cong\tilde{U}'_v$. Hence $f^{-1}(V)\cong f'^{-1}(V') $ and $ \Sigma_1\cong\Sigma$. 
	
	By Proposition  \ref{pert}, the constructed covering $ f $ is unique up to $(V,\mathrm{cr}  f_0)$-equivalence.
\end{proof}
  By Proposition \ref{R-H}, for  $[f]\in\bar{X}_{\Sigma,n}$ the sum of  multiplicities of critical values  
  is $2n-\chi(\Sigma)$. We extend the mapping $ LL $ (see the introduction) to $ \bar{X}_{\Sigma,n} $ as follows. For $ k=2n-\chi(\Sigma) $ we define $ LL:\bar{X}_{\Sigma,n} \rightarrow P^k$ by setting $LL([f])$ to be the set of all critical values of $ f $ taken with their multiplicities.

A perturbation domain $ V $ of $ f $ determines a neighbourhood of each critical value of $ f $ and  thus, defines a neighbourhood $ W $ of the point $ LL[f] $.
\begin{lem}
	\label{V}
	$W=LL(U_{[f],V})$.	     
\end{lem}
\begin{proof}
	By definition, $ LL[f']\in W$ for any $[f']\in U_{[f],V}$.
	
	For the inverse inclusion we choose $ x\in W$ and  construct a covering $ f' $ with  $ LL[f']=x$. 
	By Proposition \ref{nspert}, there exists a covering $\tilde{f}$ with $[\tilde{f}]\in U_{[f],V}$ such that all of its critical values are simple. For any critical value
	 $ w $ of $ f$ let $ \bar{U}_w$ be a connected component of $ V $, which is a closed neighbourhood of $w$. By the definition of $ W $, it follows that all  points of  $x\cap\bar{U}_w=\{x_1,\ldots,x_p\}$ lie in the interior of  $ \bar{U}_w$ and that the sum $\sum_{i=1}^p \nu_i $ of their multiplicities equals the multiplicity $ \nu_w $ of $ w $. On the other hand, the set $M_w$ of critical values of  $\tilde{f}$ that are in $ \bar{U}_w$ consists of $ \nu_w $ interior points of $ \bar{U}_w$. So we can choose disjoint closed topological disks  $ D_1,\ldots,D_p\subset \bar{U}_w$ such that $ \mathrm{Int}D_i $ contains the point $ x_i $ and $ \nu_i $ points of $M_w$. By applying  Proposition \ref{degen} to the obtained partition of the set of critical values of $\tilde{f}$, we obtain the desired covering $ f' $. By the proof of Proposition \ref{degen},  the critical values of $ f'$ with their multiplicities can be chosen at the points of $ x $. 	
\end{proof}
\begin{pr}
	\label{LL} 
	The map $ LL $ is surjective, continuous, open and finite.    
\end{pr}
\begin{proof}
	An element $x\in P^k$ determines a set of points of the sphere.
	For any point $w$ of this set, we choose a closed disk $\bar{U}_w\subset S^2$ containing $w$ as an interior point.
	We choose these disks to be pairwise disjoint. In each disk $\bar{U}_w$ we pick $mult_w$
	distinct interior points where $mult_w$ is the multiplicity of $w$ in $x$.  By \cite[Proposition 1.2.15]{LZ} and Riemann-Hurwitz formula, there exists $[f]\in X_{\Sigma,n}$ whose critical values are the picked points. The disks  $\bar{U}_w$ define  a partition of the set of   critical values. Then the surjectivity of $LL$ follows from Proposition \ref{degen}.	
	
	By Proposition \ref{base}, the family of all perturbation domains of $ f $ determines a local base of neighbourhoods $ \{U_{[f],V}\}_V$ at the point $ [f] $. 
	By Lemma \ref{V}, the set $LL(U_{[f],V})=W$ is a neighbourhood of  $ LL[f] $.  
	It is clear that $\{ LL(U_{[f],V})\}_V$ is a local base  at the point  $ LL[f] $.
	Thus $ LL $ is continuous at any point $ [f]\in\bar{X}_{\Sigma,n}$ and is open.
	
The preimage of a fixed set of critical values under the map $ LL $ consists of the coverings whose isomorphism classes are determined by degenerations of $ \Sigma $ and by collections of indices of critical points.The number of these data is finite. So $ LL $ is finite.  
\end{proof}

\begin{te}
	\label{T2}
	The space $\bar{X}_{\Sigma,n} $ is Hausdorff.	
\end{te}
\begin{proof}
	If $ LL[f_1]\neq LL[f_2] $ then the preimages of disjoint neighbourhoods of $ LL[f_1], LL[f_2] $ are disjoint neighbourhoods of $ [f_1], [f_2]$. 
	
	Suppose that $ LL[f_1]= LL[f_2] $, i.e. the sets of critical values (with their multiplicities) of
	$ f_1,f_2$ coincide. Let  $ V $ be  a common perturbation domain of $ f_1, f_2$. Choose an element $ [f']\in U_{[f_1],V}\cap U_{[f_2],V} $. Each connected component $\tilde{U}'$ of $f'^{-1}(V)$ gives a pair of homeomorphic components $ U_{v_1}\cong U_{v_2} $ of $f_1^{-1}(V)$ and $f_2^{-1}(V)$, where $ v_i $ is the center of the wedge of disks $ U_{v_i}$.
	Since  $ f $ is a common perturbation of $f_1$ and $f_2$, it follows from the definition of local perturbation that $\mu_{v_1}=\mu_{v_2}$.
	We have also $$f_1|_{\Sigma_1\setminus f_1^{-1}(\mathrm{Int}V)}=f'|_{\Sigma'\setminus f'^{-1}(\mathrm{Int}V)}=f_2|_{\Sigma_2\setminus f_2^{-1}(\mathrm{Int}V)},$$ where  $\Sigma_1,\Sigma_2,\Sigma'$ are the corresponding covering surfaces. Hence $f_1|_{\partial U_{v_1}}=f_2|_{\partial U_{v_2}}$. Thus the indices of the critical points $ v_1$ and $v_2 $ coincide, so that $f_1|_{U_{v_1}}=f_2|_{U_{v_2}}$ if we identify the surface  $ U_{v_1}$   with $ U_{v_2} $.  
	Hence $ [f_1]=[f_2] $. It proves that for $ [f_1]\neq[f_2] $, the neighbourhoods $ U_{[f_1],V} $ and $ U_{[f_2],V} $ are disjoint.
\end{proof}
\begin{sle}
	\label{W}
For	$ W=LL(U_{[f],V}) $  and $w=LL[f] $, the set $ U_{[f],V} $ is a connected component of $LL^{-1}(W)$and
we have $U_{[f],V} \cap LL^{-1}(w)=\{[f]\}$.  
\end{sle}
\begin{proof}
	By Proposition \ref{degen}, we have  $LL^{-1}(W)\subset \bigcup_{[f']\in LL^{-1}(w)}U_{[f'],V}$. The  inverse inclusion follows from Lemma \ref{V}. The members of the above union are pairwise disjoint
	(see the proof of Theorem \ref{T2}).  Thus, being open sets they are connected components of $LL^{-1}(W)$.
\end{proof}
\section{Moduli space of rational functions}
\label{rationf}
Let $F_n$ be the space of complex rational functions $f:\Cb\mathrm{P}^1\rightarrow \Cb\mathrm{P}^1$ of degree $n$,
i.~e., rational fractions $f=a(x)/b(x)$, where $a$ and $b$ are coprime homogeneous polynomials
of degree $n$. We endow $F_n$ with the topology defined by the coefficients of the polynomials, thus
we consider $F_n$ as a subspace of $\Cb\mathrm{P}^{2n+1}$.

We consider the following (right) action of group $G=PGL(2,\Cb)$ on
$F_n$. If $x=(x_0,x_1)$ is a row, then for $f=a(x)/b(x)$ and $g\in G$ we define the rational function $f^g(x)$
as $a(xg)/b(xg)$. Let $\mathcal{X}_n$ be the quotient space $F_n/G$ and let $pr:F_n\rightarrow\mathcal{X}_n$
be the quotient mapping.

Since the action of $G$ on $F_n$ is continuous, $pr$ is an \emph{open mapping}.

\begin{lem}
	\label{sequence}
	$\mathcal X_n$ is a first-countable space.
	
	Let $l$ be an arbitrary point of $\mathcal X_n$.
	Any sequence in $\mathcal X_n$ convergent to $l$ has a subsequence which is the image of
	a sequence in $F_n$ convergent to some element of $pr^{-1}(l)$.	
\end{lem}
\begin{proof}
	The mapping $pr$ is open, hence it maps any countable local base at a point
	$f\in F_n$ to a countable local base at $pr(f)$. We fix such two bases.
	
	Let $l_m\rightarrow l$ be a convergent sequence in $\mathcal{X}_n$ and let $f$ be a point in $pr^{-1}(l)$.
	For any element $U$ of the chosen countable local base  at $f$, we choose a member $l_{m_U}$ of the sequence $(l_m)$
	such that its number $n_U$ is minimal under the condition that $l_{m_U}\in pr(U)$. Then we choose
	$f_{m_U}\in pr^{-1}(l_{m_U})\cap U$. It is clear that $f_{m_U}\rightarrow f$.
\end{proof}

\begin{lem}
	\label{!lim}
	Any convergent sequence in $\mathcal{X}_n$ has a unique limit.
\end{lem}
\begin{proof}
	Let $(z_m)$ be a convergent sequence in $\mathcal{X}_n$ and let $l$, $l^*$ be its limits.
	Any  subsequence of $(z_m)$ has the same limits. Hence, by applying twice Lemma \ref{sequence}, we obtain two
	convergent sequences
	$f_m=a_m/b_m\rightarrow f=a/b$ and $f_m^*=a_m^*/b_m^*\rightarrow f^*=a^*/b^*$ in $F_n$ such that
	$pr(f_m)=pr(f_m^*)=z_m$, $pr(f)=l$, and $pr(f^*)=l^*$.
	The condition $pr(f_m)=pr(f_m^*)$ means that there exist elements
	$g_m\in G$ such that $f_m^{g_m}=f_m^*$.	
	
	Let us choose a root $\alpha$ of $a$.
	Since $a_m\to a$, we can choose a root $\alpha_m$ of each $a_m$ so that $\alpha_m\to\alpha$.
	Similarly we choose sequences of roots $\beta_m$ and $\gamma_m$ 
	of $b_m$ and $a_m+b_m$ convergent to some roots $\beta$ and $\gamma$ of $b_m$ and $a_m+b_m$
	respectively.
	We set $\alpha^*_m=\alpha_m g_m^{-1}$, $\beta^*=\beta_m g_m^{-1}$, and $\gamma^*=\gamma_m g_m^{-1}$.	
	Since these are roots of $a^*_m$, $b^*_m$, and $a^*_m+b^*_m$ respectively,
	we may assume (maybe, after passing to a subsequence of $(f_m)$) that the sequences
	$(\alpha^*_m)$, $(\beta^*_m)$, and $(\gamma^*_m)$ converge to some roots
	$\alpha^*$, $\beta^*$, and $\gamma^*$ of  $a^*$, $b^*$, and $a^*+b^*$ respectively.
	Taking into account that the polynomial $a$ and $b$ (resp. $a^*$ and $b^*$) are coprime, the points $\alpha$, $\beta$, and
	$\gamma$ (resp. $\alpha^*$, $\beta^*$, and $\gamma^*$) are distinct. Hence
	there exists a unique element $g\in G$ such that $\alpha g=\alpha^*$, $\beta g=\beta^*$, 
	and $\gamma g=\gamma^*$.
	
	Since $a_m$ and $b_m$ are coprime, the points $\alpha_m$, $\beta_m$, $\gamma_m$ are distinct for any $m$,
	thus $g_m$ is uniquely defined by the pair of triples $(\alpha_m,\beta_m,\gamma_m)$ and
	$(\alpha^*_m,\beta^*_m,\gamma^*_m)$. They converge to the pair of triples
	$(\alpha,\beta,\gamma)$ and $(\alpha^*,\beta^*,\gamma^*)$, which uniquely defines an element $g$.	
	Therefore $g_m\to g$. Passing to the limit in the equality $f_m^{g_m}=f_m^*$, we obtain
	$f^g=f^*$. Thus $l=l^*$.
	
\end{proof}
\begin{te}
	\label{Hausd}
	The space   $\mathcal{X}_n$ is Hausdorff.
\end{te}
\begin{proof}
	It is easy to check that a topological space is Hausdorff if it is first-countable
	and admits unique limits of convergent sequences.
\end{proof}

Note that Theorem \ref{Hausd} easily follows from Mumford's theorem about stable points
(see, for example, \cite[Part II, Section 4.6, Theorem 4.16]{VP};
one can adapt the same arguments as in  Example $1^\circ$ in Section 4.6 of this book:
it is enough just to replace
the space of binary forms $PV_{2g+2}$ and the discriminant by $F_n$ and the resultant of $a$ and $b$, respectively).
Mumford's theorem implies, moreover, that $\mathcal{X}_n$ is an algebraic variety.

\subsection{Lyashko-Looijenga map}
$\mathcal{LL}$ takes a non-constant rational function $f:\Cb\mathrm{P}^1\rightarrow \Cb\mathrm{P}^1$ to the
monic polynomial $(t-t_1)^{l_1}\ldots(t-t_r)^{l_r}$
of an indeterminate $t$ whose roots are critical values of $f$ and the multiplicity $l_i$ of a root $t_i$
is equal to $\sum_{x\in f^{-1}(t_i)}(\deg_xf-1)$,  where $\deg_xf$ is  the
multiplicity of $ x $ as the root of the equation  $f(x)=t_i$ (cf.
\ref{Deform}). 

According to \cite[Lemma 4.3]{L}, the image of a rational function
$f=P(x)/Q(x)$ under $\mathcal{LL}$ is the discriminant of $P(x)- tQ(x)$ up to constant factor.

We shall need the following homogeneous analog of Lyashko-Looijenga map
$$
\mathcal{LL}[P(x):Q(x)]=\mathrm{discr}_x(t_1Q(x)-t_0P(x)).
$$
Recall that the discriminant of a homogeneous polynomial $F(x_0,x_1)$ of degree $d$
can be defined as the resultant of $F'_{x_0}(x_0,1)$ and $F'_{x_1}(x_0,1)$.
It is a homogeneous polynomial of degree $2d-2$ in the coefficients of $F$ (see, e.~g., \cite[\S 1]{Cool}).

It is clear that the mapping $\mathcal{LL}$ is invariant under the action of $G$,
hence  it factors through the quotient: $\overline{\mathcal{LL}}\circ~pr=\mathcal{LL}$.

\begin{za}
	\label{OpenLL}
	Since the mapping $\mathcal{LL}$ is polynomial and $pr$ is open, the mapping
	$\overline{\mathcal{LL}}$ is continuous and open.
\end{za}

\subsection{Relation with branched coverings}
Let $X_n=X_{S^2,n}$, where we identify the target (covered) sphere with $\Cb\mathrm{P}^1$. Let
$c:X_n\rightarrow\mathcal{X}_n$ be the mapping which takes $[f]$ to the orbit
$[f^c]=Gf^c$ of a rational function $f^c$ obtained by pulling back the complex structure from
$\Cb\mathrm{P}^1$ to the covering sphere $S^2$.

\begin{te}
	\label{Topol}
	The mapping $c$	is a homeomorphism.
\end{te}
\begin{proof}
	It is well known (see, e.~g., \cite[\S 1.8]{LZ}) that $c$ is a bijection. To prove that it is
	a homeomorphism, it is enough to show that $c$ maps a local base  at any point $[f]\in X_n$
	to a local base  at $[f^c]\in\mathcal{X}_n$.
	We identify $P^{2n-2}$ (recall that $P^k$ is the $k$-th symmetric power of $S^2$)
	with the space of homogeneous polynomials of degree $2n-2$.
	It is clear that $\overline{\mathcal{LL}}\circ c=LL$. Therefore $\overline{\mathcal{LL}}$ is of finite covering degree
	because $LL$ is so
	(see \ref{LL}). Since $\overline{\mathcal{LL}}$ is finite and $\mathcal{X}_n$ is Hausdorff,
	there exists a neighbourhood $W$ of  $\overline{\mathcal{LL}}[f^c]$ such that the connected
	component  $U$ of $\overline{\mathcal{LL}}^{-1}(W)$ is a neighbourhood of  $[f^c]$.
	Since $\overline{\mathcal{LL}}$ is open and continuous, $\overline{\mathcal{LL}}|_U^{-1}$
	lifts any local base  at $\overline{\mathcal{LL}}[f^c]$ lying in $W$ to a local base 
	at $[f^c]$. By Lemma \ref{V}, we can choose $LL(U_{[f],V})\subset W$ as the neighbourhoods of the lifted base.
	By combining this facts with Corollary \ref{W}, we deduce that
	$\{c(U_{[f],V})\}_V$ is a local base  at $[f^c]$.
\end{proof}

\section{Natanzon--Turaev compactification}
\label{NTc}
Let $\Sigma$ be a closed oriented surface and let $H(\Sigma,n)\subset X_{\Sigma,n}$ be the space of
isomorphism classes of orientation preserving $n$-fold branched coverings $f:\Sigma\rightarrow S^2$
with simple critical values.
Here we describe the compactification $N(\Sigma,n)$ of $H(\Sigma,n)$ constructed by Natanzon and
Turaev \cite{NT}.

A \emph{decorated function} on $\Sigma$ is defined as a triple $(f,E,\{D_e\}_{e\in E})$, where
$f\in H(\Sigma,n)$, $E$ is a finite subset of $S^2$ which does not contain critical values of $f$,
and $\{D_e\}_{e\in E}$ is a collection of pairwise disjoint disks in $S^2$ such that for any $e\in E$:
\begin{itemize}
	\item the interior of $D_e$ contains $e$ and at least two critical values of $f$;
	\item the boundary of $D_e$ does not contain critical values of $f$.
\end{itemize}
The \emph{multiplicity} of $e\in E$ is defined as the number of critical values of $f$ belonging to $D_e$.

Two decorated functions $(f,E,\{D_e\}_{e\in E})$ and $(f',E',\{D'_e\}_{e\in E'})$ are called
\emph{equivalent} if $E=E'$ and there exist homeomorphisms
$\alpha:\Sigma\rightarrow\Sigma$ and $\varphi:S^2\rightarrow S^2$ such that
\renewcommand{\labelenumi}{(\theenumi)}
\renewcommand{\theenumi}{\roman{enumi}}
\begin{enumerate}
	\item $\varphi\circ f=f'\circ\alpha$;
	\item $\varphi$ keeps fixed all points of $E$ and those critical values of $f$ which are not in $\bigcup_eD_e$;
	\item $\varphi$ is isotopic to the identity mapping in the class of homeomorphisms satisfying (ii);
	\item $\varphi(D_e)=D'_e$ for each $e\in E$. 
\end{enumerate}
We denote the equivalence class of a decorated function $(f,E,\{D_e\}_{e\in E})$ by $[f,E,\{D_e\}_{e\in E}]$.

Let $N(\Sigma,n)$ be the set of equivalence classes of the decorated functions.

The topology in $N=N(\Sigma,n)$ is introduced in terms of neighbourhoods as follows.
Given $x\in N$ and a decorated function $(f,E,\{D_e\}_{e\in E})$ belonging to $x$,
let $\bar E$ be the union of $E$ with the set of critical values lying outside $\bigcup_{e\in E}D_e$.
For each $e'\in\bar E\setminus E$ we choose a closed disk $D_{e'}\subset S^2\setminus(\bigcup_{e\in E}D_e)$
containing $e'$ in its interior. We do it so that the disks $D_{e'}$ are pairwise disjoint.
Let $D=\{D_e\}_{e\in\bar E}$. Then the \emph{neighbourhood} $U_D$ \emph{of} $x\in N$
is defined as the set of all classes of decorated functions $(f',E',\{D'_{e'}\}_{e'\in E'})$ with
$\bigcup_{e'}D'_{e'}\subset\bigcup_{e\in E}D_e$ such that $f'=h\circ f$, where $h$
is a homeomorphism of the sphere fixed outside $\bigcup_{e\in\bar E}\mathrm{Int}D_e$.

In the following evident remark we use the notation of the last definition.
\begin{za}
	\label{VE}     
	Decorated functions $(f,E,\{D_e\}_{e\in E})$ and $(f',E,\{D_e\}_{e\in E})$
	are equivalent if and only if the coverings $f$ and $f'$ are $(V,\bar E)$-equivalent
	(see Section \ref{Deform}), where $V=\bigcup_{e\in\bar E}D_e$. \qed
\end{za}
Recall that $P^k$ is the $k$-th symmetric power of $S^2$.
Let $q:N(\Sigma,n)\rightarrow P^k$, $k=2n-\chi(\Sigma)$,
be such that
$q[f,E,\{D_e\}_{e\in E}]$ is the collection of critical values of $f$ lying outside
$\bigcup_{e\in E}D_e$ taken with multiplicity $1$ and points $e\in E$ taken with the multiplicities
defined above.
\begin{te}
	\label{NT}
	There exists a unique homeomorphism $\nu: \bar{X}_{\Sigma,n}\rightarrow N(\Sigma,n)$ fixed on $ H(\Sigma,n)$
	and such that $q\circ\nu=LL$.    
\end{te}

\begin{proof}
	For $[f]\in \bar{X}_{\Sigma,n}$ we denote the set of all critical values of $f$ by $\bar E= \mathrm{cr} f$
	and we denote the set af multiple critical values of $f$ by $E$.
	Let $V=\bigcup_{e\in\bar E} D_e$ be a perturbation domain of $f$.
	By Proposition \ref{nspert}, there exists a perturbation
	$\tilde{f}$ of $f$ with $[\tilde{f}]\in U_{[f],V}\cap H(\Sigma,n)$. Moreover, by Proposition \ref{pert}
	we may assume that $E\cap\mathrm{cr}\tilde{f} = \varnothing$ and $\bar E\setminus E\subset\mathrm{cr}\tilde{f}$.
		Then we set $\nu[f]=[\tilde{f},E,\{D_e\}_{e\in E}]$.  
	
	Let us show that the definition of $\nu[f]$ does not depend on the choices we are done.
	Let $[f']=[f]$. Then we have $\mathrm{cr} f' = \bar E$. Let $V' = \bigcup_{e\in\bar E} D'_e$ be a domain of
	perturbation of $f'$ and
	let $\tilde f'$ be a perturbation of $f'$ such that $[\tilde{f}']\in U_{[f'],V'}\cap H(\Sigma,n)$.
	As above, we may assume that $E\cap\mathrm{cr}\tilde f' = \varnothing$ and $\bar E\setminus E\subset\mathrm{cr}\tilde f'$.
	There exists an isotopy $\{\varphi_t:S^2\to S^2\}_{t\in[0,1]}$
	fixed on $\bar E$ such that $\varphi_0=\text{id}_{S^2}$ and such that $\varphi_1(D'_e)=D_e$ for any
	$e\in\bar E$. Therefore
	$[\tilde{f}',E,\{D'_e\}_{e\in E}]=[\varphi_1\circ \tilde{f}',E,\{D_e\}_{e\in E}]$
	and $[\varphi_1\circ \tilde{f}']\in U_{f',V}\cap H(\Sigma,n) = U_{f,V}\cap H(\Sigma,n)$.
	By Proposition \ref{nspert} the coverings $\varphi_1\circ\tilde f'$ and $\tilde f$ are
	$(V,\tilde E)$-equivalent. Hence
	$[\varphi_1\circ \tilde{f}',E,\{D_e\}_{e\in E}] = [\tilde{f},E,\{D_e\}_{e\in E}]$
	according to Remark \ref{VE}. Thus $\nu$ is well defined.

	Let us define the inverse mapping.
	Let $(\tilde{f},E,\{D_e\}_{e\in E})$ be a decorated function
	and let $\bar{E}$ be the union of $E$ with the set of simple critical values of $\tilde{f}$
	which do not belong to $\bigcup_{e\in E}D_e$.
	In $S^2\setminus(\bigcup_{e\in E}D_e)$, we choose pairwise disjoint closed disk neighbourhoods
	$\{D_e\}_{e\in \bar{E}\setminus E}$ of points of $\bar{E}\setminus E$.
	Let $V=\bigcup_{e\in \bar{E}}D_e$. By Proposition \ref{degen}, there exists a degeneration
	$f$ of $\tilde{f}$ such that $[\tilde{f}]\in U_{[f],V}\cap H(\Sigma,n)$ and $ LL(f)=q(\tilde{f})$.
	It is easy to check that the mapping $\nu^{-1}:[\tilde{f},E,\{D_e\}_{e\in E}]\mapsto[f]$
	is well defined (here we apply the definition of isomorphism of branched coverings given in Section \ref{Deform})
	and that it is the inverse of $\nu$. Hence $\nu$ is bijective.
	
	Let us show that $\nu$ maps the local base  at any point $[f]$ to a local base 
	at $\nu[f]$. Let, as above, $\bar E$ (resp. $E$) be the set of all (resp. of multiple) critical values of $f$,
	let $V=\bigcup_{e\in\bar E} D_e$ be a perturbation domain for $f$, and let
	$\nu[f]=[\tilde{f},E,\{D_e\}_{e\in E}]$.
	Let $[f']$ be an arbitrary element of $U_{[f],V}$ and let
	$\nu[f']=[\tilde{f'},E',\{D'_{e'}\}_{e'\in E'}]$, where
	$V'=\bigcup_{e'\in\bar E'} D'_{e'}$ is
	a perturbation domain of $f'$ which is chosen so that $V'\subset V$. 
	Then we have $\bigcup_{e'\in E'}D'_{e'}\subset\bigcup_{e\in E}D_e$. 
	We are going to show that $\nu[f']$ belongs to the neighbourhood $U_D$ of $\nu[f]$, where
	$D=\{D_e\}_{e\in\bar E}$. To this end it is enough to check that
	$[\tilde{f'}]=[h\circ\tilde{f}]$, where
	$h:S^2\to S^2$ is a homeomorphism identical outside
	$\bigcup_{e\in\bar E}\mathrm{Int}D_e=\mathrm{Int}V$.
	By definition, $[\tilde{f}]\in U_{[f],V}\cap H(\Sigma,n)$. Since
	$V'\subset V$, we have $[\tilde{f'}]\in U_{[f'],V'}\cap H(\Sigma,n)\subset U_{[f],V}\cap H(\Sigma,n)$
	by the transitivity of perturbation, see Remark \ref{trans}.
	Hence, by Proposition \ref{nspert}, the coverings $\tilde{f}$ and $\tilde{f'}$ are $(V,\bar E)$-equivalent,
	thus there exist homeomorphisms $\alpha:\Sigma\rightarrow \Sigma$ and $h:S^2\rightarrow S^2$ such that
	$h\circ\tilde{f}=\tilde{f'}\circ\alpha$.
	By replacing $\tilde{f'}$ with the isomorphic covering $\tilde{f'}\circ\alpha$ we obtain
	$[\tilde{f'}]=[h\circ\tilde{f}]$. Therefore $\nu(U_{[f],V})\subset U_D$. 
	
	Conversely, let $U_D$ be a neighbourhood of $[\tilde{f},E,\{D_e\}_{e\in E}]$
	defined by a family of closed disks $D=\{D_e\}_{e\in\bar E}$ and let
	$[\tilde{f'},E',\{D'_{e'}\}_{e'\in E'}]\in U_D$. We set $V=\bigcup_{e\in\bar E}D_{e}$ and
	$V'=\bigcup_{e'\in E'}D'_{e'}\cup\bigcup_{e\in\bar E\setminus E}D_{e}$.
	Then $ V'\subset V $ by the definition of $ U_D$.
	Hence, for $[f]=\nu^{-1}[\tilde{f},E,\{D_e\}_{e\in E}]$ and $[f']=\nu^{-1}[\tilde{f'},E',\{D'_{e'}\}_{e'\in E'}]$,
	we obtain $[f']\in U_{[f],V}$ by the transitivity of perturbation (see Remark \ref{trans}).
	Thus $\nu^{-1}(U_D)\subset U_{[f],V}$.
	
	By Proposition \ref{base} combined with \cite[Theorem 1.5]{NT}, the families
	$\{U_{[f],V}\}_V$ and $\{U_D\}_D$ are local bases  at $[f]$ and $\nu[f]$ respectively.
	Hence $\nu$ is a homomorphism whose uniqueness follows from its construction.
	The equality $q\circ\nu=LL$ immediately follows from the definitions.
\end{proof}

If $S^2$ is endowed with the complex structure, then its symmetric power $P^k$ becomes a
projective space. 
Since $q|_{H(\Sigma,n)}$ is an unramified covering over its Zariski open subset,
$H(\Sigma,n)$ endowed with the complex structure lifted from $P^k$ is a smooth
quasiprojective variety.

\begin{za}
	\label{norm} 
	The space $H(\Sigma,n)$ is connected \emph{(see \cite{H}; note that when Hurwitz formulates in
		\cite{H} this result, he cites an earlier Clebsch' paper)}, hence $N(\Sigma,n)$ is connected as well.
	\emph{Diaz \cite[p. 2]{D} has shown that} $N(\Sigma,n)$ is homeomorphic to a normal projective variety.
	Since this variety is normal, its singularities are of complex codimension at least two.
\end{za} 
The authors are grateful to the referee for indicating some gaps and mistakes in the first version of the paper.


\end{document}